\documentclass[11pt, reqno]{amsart}

\usepackage{amsmath}
\usepackage{amsthm}
\usepackage{amsfonts}
\usepackage{amssymb}

\usepackage{enumitem}
\setitemize{leftmargin=*} 

\newtheorem{lemma}{Lemma}[section]

\newtheorem{theorem}[lemma]{Theorem}
\newtheorem{corollary}[lemma]{Corollary}
\theoremstyle{remark}
\newtheorem{remark}[lemma]{Remark}

\theoremstyle{definition}
\newtheorem{definition}[lemma]{Definition}

\numberwithin{equation}{section}

\newcommand{\mb}{\mathbf}
\newcommand{\mc}{\mathcal}

\newcommand{\R}{{\mathbb R}}
\newcommand{\N}{{\mathbb N}}

\newcommand{\C}{{\mathbb C}}
\newcommand{\rst}[1]{\ensuremath{{\mathbin |}%
\raise-.5ex\hbox{$#1$}}} 

\newcommand{\bs}[1]{\boldsymbol{#1}}

\title[Stable blow up dynamics]{Stable blow up dynamics for energy supercritical wave equations}

\author{Roland Donninger}
\address{\'Ecole Polytechnique F\'ed\'erale de Lausanne, 
Department of Mathematics, Station 8, CH-1015 Lausanne, Switzerland}
\email{roland.donninger@epfl.ch}

\author{Birgit Sch\"orkhuber}
\address{Vienna University of Technology, Institute for Analysis and Scientific Computing,
Wiedner Hauptstra\ss e 8-10,  A-1040 Vienna, Austria}
\email{birgit.schoerkhuber@tuwien.ac.at}
\thanks{The second author acknowledges partial support from the Austrian Science Fund
(FWF), grants P22108, P23598, P24304, and I395; the Austrian-French Project of
the Austrian Exchange Service (\"OAD); and the Innovative Ideas
Program of Vienna University of Technology.}

\begin{document}

\maketitle

\begin{abstract}
We study the semilinear wave equation
\[ \partial_t^2 \psi-\Delta \psi=|\psi|^{p-1}\psi \]
for $p > 3$ with radial data in three spatial dimensions.
There exists an explicit solution which blows up at $t=T>0$ given by
\[ \psi^T(t,x)=c_p (T-t)^{-\frac{2}{p-1}} \]
where $c_p$ is a suitable constant.
We prove that the blow up described by $\psi^T$ is stable in the sense that there
exists an open set (in a topology strictly stronger than the energy) of radial initial data 
that lead to a solution which converges to $\psi^T$ as $t\to T-$ in the backward lightcone of
the blow up point $(t,r)=(T,0)$.
\end{abstract}

\section{Introduction}

We consider the Cauchy problem for the focusing semilinear wave equation
\begin{equation}
\label{eq:main} 
\partial_t^2 \psi -\Delta\psi=|\psi|^{p-1}\psi 
\end{equation}
for $\psi: I\times \R^3\to\R$, $I$ an interval, and $p>3$ fixed.
The conserved energy $\mc E$ associated to Eq.~\eqref{eq:main} is given by
\[ \mc E(\psi(t,\cdot),\partial_t \psi(t,\cdot))=\tfrac12\|(\psi(t,\cdot),
\partial_t \psi(t,\cdot))\|_{\dot H^1\times L^2(\R^3)}^2
-\tfrac{1}{p+1}\|\psi(t,\cdot)\|_{L^{p+1}(\R^3)}^{p+1}. \]
Following the usual terminology we call $\mc E(\psi(t,\cdot),\partial_t \psi(t,\cdot))$ the total energy and
$\frac12 \|(\psi(t,\cdot),\partial_t \psi(t,\cdot))\|_{\dot H^1\times L^2(\R^3)}^2$ the (free) energy of $\psi$.
Eq.~\eqref{eq:main} is invariant under the scaling transformation
\begin{equation}
\label{eq:scaling} 
\psi(t,x)\mapsto \psi_\lambda(t,x)
:=\lambda^{-\frac{2}{p-1}}\psi(\tfrac{t}{\lambda},\tfrac{x}{\lambda}),\quad \lambda>0 
\end{equation}
and the energy scales as 
\[ \mc E(\psi_\lambda(t,\cdot),\partial_t \psi_\lambda(t,\cdot))=
\lambda^{\frac{p-5}{p-1}}\mc E(\psi(\tfrac{t}{\lambda},\cdot),\partial_1 \psi(\tfrac{t}{\lambda},\cdot)). \]
This shows that Eq.~\eqref{eq:main} is energy subcritical for $1<p<5$, critical for $p=5$, and
supercritical for $p>5$.

It is well-known that Eq.~\eqref{eq:main} exhibits finite-time blow up for data with negative
energy \cite{Lev74}.
A more explicit way to obtain information on the blow up behavior is to look for self-similar
solutions, i.e., solutions which are invariant under the natural scaling \eqref{eq:scaling}.
After a time translation those solutions are of the form
\[ \psi(t,x)=(T-t)^{-\frac{2}{p-1}}f(\tfrac{x}{T-t}) \]
for a function $f: \R^3\to \R$ where the free constant $T>0$ is called the blow up time.
One expects that Eq.~\eqref{eq:main} admits many smooth self-similar solutions, even in the
radial context, see \cite{BizMaiWas07}.
The simplest self-similar solution is given by
\[ \psi^T(t,x):=\kappa_p^{\frac{1}{p-1}} (T-t)^{-\frac{2}{p-1}},\]
where $\kappa_p = \frac{2(p+1)}{(p-1)^2}$, i.e., in this case the function $f$ is just a 
constant.
We call $\psi^T$ the \emph{fundamental self-similar} or \emph{ODE blow up solution}.

In what follows we restrict ourselves to radial solutions and write $\psi(t,r)$ where
$r=|x|$.
It is natural to study self-similar solutions in the backward lightcone 
$\mc C_T:=\{(t,r): t\in (0,T), r\in (0,T-t)\}$ of the blow up
point $(T,0)$.
As a matter of fact, the Cauchy problem for the wave equation \eqref{eq:main} restricted
to $\mc C_T$ is a dynamical system of its own, decoupled and completely independent of 
the behavior outside of $\mc C_T$.
This is immediate by basic domain of dependence considerations and of great conceptual
importance for our approach.
Interestingly enough, the wave equation restricted to $\mc C_T$ displays many features
which are reminiscent of parabolic systems, e.g.~the problem is only well-posed
in forward time. 
Numerical studies \cite{BizChmTab04} suggest that $\psi^T$ describes the generic blow up
of the system.
More precisely, it is demonstrated numerically in \cite{BizChmTab04} that sufficiently large ``generically''
chosen data lead to a time evolution which converges to $\psi^T$ in $\mc C_T$ as $t\to T-$.
A first step in approaching this problem from a rigorous perspective is to study the stability
of $\psi^T$ which is the content of the present paper. 
In this respect it is illustrative to consider the (free) energy of $\psi^T$ in $\mc C_T$.
By scaling it is easy to see that
\[ \|(\psi^T(t,\cdot), \partial_t \psi^T(t,\cdot))\|_{\dot H^1\times L^2(B_{T-t})}^2\simeq (T-t)^{\frac{p-5}{p-1}} \]
where $B_{T-t}:=\{x\in\R^3: |x|<T-t\}$.
Consequently, in the supercritical case $p>5$, the singularity formation described by $\psi^T$
is not energy-trapping.
This suggests that self-similar blow up in the supercritical regime cannot be studied
in the energy topology.
Consequently, the choice of a suitable stronger topology is a crucial step in our construction.
Our main result may be formulated qualitatively as follows, see Theorem \ref{Th:Main} below
for the precise statement.

\begin{theorem}[Main theorem, qualitative version]
\label{thm:main1}
For any $p>3$ there exists an open set (in a suitable topology stronger than the energy) of
radial initial data such that the corresponding solution of Eq.~\eqref{eq:main} converges to
$\psi^T$ in $\mc C_T$ as $t\to T-$ where $T$ is a suitable blow up time, depending on the data.
In this sense, the blow up described by $\psi^T$ is stable.
\end{theorem}

\subsection{Brief history of the problem}
Needless to say that the semilinear wave equation \eqref{eq:main} has been the subject
of many studies and it is impossible to review the entire literature. 
Consequently, we focus on recent developments which are related to our work.
In particular the critical case $p=5$ attracted a lot of interest in
the recent past. For $p=5$ there exists a soliton solution $W$ which is the central object
of many works. On the one hand, the energy of $W$ yields a threshold for global existence, see
\cite{KenMer08}, \cite{DuyKenMer08}. 
On the other hand, rescalings of $W$ can be used to construct exotic solutions
\cite{KriSchTat09}, \cite{DonKri12}.
Furthermore, dynamics around the soliton were studied, see e.g.~\cite{KriSch07}, 
\cite{KriNakSch10}, \cite{KriNakSch11}, and remarkable classification theorems were proved
\cite{DuyKenMer11}, \cite{DuyKenMer12a}, \cite{DuyKenMer12b}, \cite{DuyKenMer12c}.
However, as far as the role of self-similar blow up is concerned, previous results
are mainly confined to the case $p\leq 3$, see, however, the very recent \cite{KilStoVis12}
which establishes blow up bounds in the entire subcritical regime $1<p<5$.
In \cite{MerZaa03}, \cite{MerZaa05} it is proved that any blow up solution
of Eq.~\eqref{eq:main} with $p\leq 3$ blows up at a universal rate which is given by
$\psi^T$, see also \cite{MerZaa05a}.
Moreover, profile convergence to $\psi^T$, i.e., the analogue of Theorem \ref{thm:main1} in
the (sub)conformal case $p\leq 3$, was recently proved by the authors \cite{DonSch12}.
We further remark that the corresponding problem in one space dimension is well-understood
\cite{MerZaa07}, \cite{MerZaa08}, \cite{CotZaa11}.
For the supercritical regime $p>5$, which is the main concern of the present paper, much less is known.
In general, the study of supercritical wave equations is only at its beginnings, see 
e.g.~\cite{KenMer11}, \cite{KilVis11}, \cite{Bul11}, and references therein for recent progress in the field.
Furthermore, analogues of Theorem \ref{thm:main1} were proved for the supercritical wave maps
and Yang-Mills problems \cite{Don11}, \cite{Don12}.
We would also like to emphasize that the proof of our main result involves the construction of large data solutions for
supercritical wave equations starting from an open set of initial data. 
In view of the ongoing efforts to understand large data dynamics in supercritical equations,
we believe that our result is also interesting from this perspective.

\subsection{Outline of the proof}
The proof is based on the general stability theory for self-similar solutions 
developed in \cite{DonSchAic12}, \cite{Don11}, \cite{DonSch12}, \cite{Don12}.
We work exclusively in self-similar coordinates $\tau=-\log(T-t)$, $\rho=\frac{r}{T-t}$
and restrict ourselves to the evolution in the lightcone $\mc C_T$.
Furthermore, we do not rely on any previous well-posedness theory for the wave equation.
In our approach, the required existence and uniqueness results follow automatically.
The main philosophy of our method is to take the self-similar coordinates as a starting point and to develop
the entire theory of the wave equation in these variables.
The main technical difficulty comes from the fact that the introduction of the coordinates $(\tau,\rho)$
destroys the self-adjoint structure of the problem which precludes the application of
standard spectral methods.
Instead, we rely on semigroup theory to study the evolution.
A first crucial step in this respect consists of identifying a suitable Hilbert space structure
that automatically yields sharp decay estimates for the free equation in the lightcone
$\mc C_T$.
This is the first point where we crucially depart from the proof of the (sub)conformal result \cite{DonSch12}
since for the latter the energy was sufficient for this purpose.
In the present work we have to require more regularity.
In a next step we linearize the problem around $\psi^T$.
A detailed spectral analysis and abstract results from semigroup theory yield
an almost sharp decay estimate for the linearized evolution.
The point here is that the clever choice of a Hilbert space allows us to avoid almost
any ``hard'' analysis of the resolvent and we can rely on abstract results.
This comes at the affordable price of an $\varepsilon$-loss
in the final decay estimate.
It is important to note that the obtained decay is always exponential. This is due to 
the self-similar coordinates.
Since exponential decay is reproduced by the Duhamel formula, the perturbative 
treatment of the nonlinear problem is in principle straightforward.
A complication arises from the fact that the blow up time $T$ is a free parameter and the
whole problem is time translation invariant.
This has to be accounted for by some kind of modulation theory.
We use an infinite-dimensional version of the Lyapunov-Perron method from dynamical systems
theory.
This means that we first modify the initial data in order to force convergence to $\psi^T$.
In a second step we then show that this modification can be removed by choosing the
blow up time $T$ accordingly. 

\subsection{Notation}

Throughout we assume $p$ to be a fixed real number with $p>3$.
For a closed linear operator $\bs L$ we write $\sigma(\bs L)$ and
$\sigma_p(\bs L)$ for the spectrum and point spectrum, respectively.
Furthermore, we set $\bf R_{\bs L}(\lambda):=(\lambda-\bs L)^{-1}$ for $\lambda \notin \sigma(\bs L)$.
The Fr\'echet derivative of a map $f$ is denoted by $Df$ and we also use the notation
$D_y f(x,y)$ for the partial Fr\'echet derivative with respect to the second variable.
As usual, $a\lesssim b$ means $a\leq cb$ for an absolute constant $c>0$ and we also write
$a\simeq b$ if $a\lesssim b$ and $b\lesssim a$.

\subsection{Higher energy norm -- local version}

We intend to study perturbations $\varphi$ of the fundamental self-similar solution $\psi^T$. 
Thus, we insert the ansatz $\psi=\psi^T+\varphi$ into Eq.~\eqref{eq:main} and obtain the
Cauchy problem
\begin{align}
\label{Eq:perturbation_fullcauchy}
\left \{ \begin{array}{l}
\varphi_{tt}-\varphi_{rr}-\frac{2}{r}\varphi_r-p(\psi^T)^{p-1} \varphi - N_T(\varphi)=0 \mbox{ in } \mc{C}_T \\
\left. \begin{array}{l} \varphi(0,r)=f(r)-\psi^T(0,r) \\
 \varphi_t(0,r)=g(r)-\psi^T_t(0,r) \end{array} \right \} \mbox{ for }r \in [0,T]
\end{array} \right .
\end{align}
for the perturbation $\varphi$ in the backward lightcone $\mc{C}_T$ of the blow up point 
$(T,0)$.
Here, $f$ and $g$ are the free initial data of the original problem and 
\[ N_T(\varphi)=|\psi^T+\varphi|^{p-1}(\psi^T+\varphi)-|\psi^T|^{p-1}\psi^T-p |\psi^T|^{p-1}\varphi \]
is the nonlinear remainder.
We intend to study equation \eqref{Eq:perturbation_fullcauchy} in a Hilbert space with an 
inner product which is associated to a conserved quantity 
of the \textit{free equation}
\begin{align}
\label{Eq:Free_wave} 
\varphi_{tt}-\varphi_{rr}-\frac{2}{r}\varphi_r=0.
\end{align}
Furthermore, we require a certain degree of regularity such that the nonlinearity is well-defined
in the respective Hilbert space.
The fact that we are considering the problem not in the whole space but on a bounded domain
introduces certain technical difficulties with respect to the choice of the function space. 
Quantities such as the energy associated to \eqref{Eq:Free_wave}  given by
\begin{align*}
 \int_0^{\infty} [ \varphi_t(t,r)^2 + \varphi_r(t,r)^2]r^2dr 
\end{align*}
do not necessarily define a \textit{local} norm due to the lack of a boundary condition at $r=0$.
In the (sub)conformal range \cite{DonSch12} we have dealt with this problem by transforming
\eqref{Eq:Free_wave} to the $1+1$ wave equation. Setting 
$\tilde \varphi := r \varphi$ yields $\tilde \varphi(t,0) = 0$ and \eqref{Eq:Free_wave} reads 
\begin{align}
\label{Eq:One_dim_wave}
\tilde \varphi_{tt} - \tilde \varphi_{rr} = 0.
\end{align}
For $p \leq 3$ the norm associated to a \textit{local} 
version of the energy of \eqref{Eq:One_dim_wave}, which is given by
\begin{align}
\label{Eq:Tilde_Energy}
 \int_0^R [ \tilde \varphi_t(t,r)^2 +  \tilde \varphi_r(t,r)^2 ]dr 
\end{align}
for some finite $R >0$, was suitable to study the problem. 
However, for $p > 3$ we need more regularity. Note that differentiating 
Eq.~\eqref{Eq:One_dim_wave} with respect to $r$ shows that $\tilde \varphi_r$ again satisfies the
one-dimensional wave equation. This immediately suggests to choose 
\[ \int_0^R [ \tilde \varphi_{rt}(t,r)^2 +  \tilde \varphi_{rr}(t,r)^2 ] dr. \]
However, $\tilde \varphi_r$ does not satisfy an appropriate boundary condition at the origin. 
Therefore, in order to obtain a local norm, we simply add the energy term \eqref{Eq:Tilde_Energy}.
This motivates the following definition.

\begin{definition}
For $R > 0$ we set $\mathcal {\tilde E}^h(R) := C^2[0,R] \times C^1[0,R]$ and define
\begin{align*}
\| (f,g) \|^2_{\mc E^h(R)}  & := \int_0^R |r f'(r) + f(r)|^2 dr + \int_0^R |rf''(r) + 2 f'(r)|^2 dr   \\
&+ \int_0^R r^2 |g(r)|^2 dr+ \int_0^R |r g'(r) +  g(r)|^2 dr.
\end{align*}
We denote by $\mathcal E^h(R)$ the completion of $\mathcal {\tilde E}^h(R)$ with respect to
$\| \cdot\|_{\mc E^h(R)}$ and refer to $(\mc E^h(R), \| \cdot\|_{\mc E^h(R)})$ as the \textit{local higher energy space}.
\end{definition}


Inserting $\psi^T$ shows that the fundamental self-similar solution blows up 
in the lightcone $\mc C_T$ 
with respect to the local higher energy norm, i.e.,
\begin{align*}
&\|(\psi^T(t,\cdot),\psi_t^T(t,\cdot))\|^2_{\mc E(T-t)}  \\ 
 & = \int_0^{T-t} |\psi^T(t,r)|^2 dr + \int_0^{T-t} |r^2 \psi_t^T(t,r)|^2 dr  + \int_0^{T-t} |\psi_t^T(t,r)|^2 dr  \\
& \simeq (T-t)^{\frac{p-5}{p-1}} + (T-t)^{-\frac{p+3}{p-1}} \simeq (T-t)^{-\frac{p+3}{p-1}}
\end{align*}
for $t\in [0,T)$.
\begin{theorem}[Main result, precise formulation]\label{Th:Main}
Fix $p\in\R$, $p > 3$. Choose $\varepsilon>0$ such that 
\[ \mu_p := \tfrac{2}{p-1}-\varepsilon  > 0 \]
and let $(f,g)$ be radial initial data with
\begin{align*}
\|(f,g) - (\psi^1(0,\cdot), \psi_t^1(0,\cdot))\|_{\mc E^h(\frac{3}{2})}
\end{align*}
sufficiently small. Then there exists a $T \in (\frac12,\frac32)$ such that the Cauchy problem
\begin{align*}
\left \{ \begin{array}{l}
\partial_t^2 \psi-\Delta \psi = |\psi|^{p-1} \psi  \\ 
\psi[0]=(f,g) 
\end{array} \right .
\end{align*}
has a unique radial solution $\psi: \mc{C}_T \to \R$ which satisfies 
\[ (T-t)^{\frac{p+3}{2(p-1)}}\|(\psi(t,\cdot),\psi_t(t,\cdot))-(\psi^T(t,\cdot),\psi^T_t(t,\cdot))\|_{\mc E^h(T-t)} 
\leq C_{\varepsilon} (T-t)^{\mu_p}\]
for all $t \in [0,T)$ and a constant $C_{\varepsilon} > 0$.
\end{theorem}
\subsection{First order formulation in similarity coordinates}

We proceed as in \cite{DonSch12}, i.e., we choose variables
$$\varphi_1=(T-t)^{\frac{2}{p-1}} (r\varphi)_t, \quad \varphi_2=(T-t)^{\frac{2}{p-1}} (r\varphi)_r$$
and transform the resulting first order system to similarity coordinates $(\tau, \rho)$ defined by
\[\rho = \frac{r}{T-t}, \quad \tau = - \log (T-t).\]
Setting
\[\phi_j(\tau,\rho) := \varphi_j(T-e^{-\tau},e^{-\tau} \rho)\] 
for $j=1,2$ and recalling that $\partial_t = e^{\tau} (\partial_{\tau} + \rho \partial_{\rho})$ 
and $\partial_r = e^{\tau} \partial_{\rho}$ the resulting system of equations read
\begin{align}\label{eq:nonlinear_firstorder_css}
\left \{ \begin{array}{l}
\left. \begin{array}{l}
\partial_\tau \phi_1  =  -\rho \partial_\rho \phi_1 + \partial_\rho  \phi_2 - \frac{2}{p-1} \phi_1    \\
\phantom{10mm} + p\kappa_p  \int_0^\rho \phi_2(\tau,s)ds + \rho N \left(\rho^{-1} \int_0^\rho \phi_2(\tau,s)ds\right)
 \\
\partial_\tau \phi_2=  -\rho \partial_\rho \phi_2 + \partial_\rho  \phi_1 - \frac{2}{p-1} \phi_2 \end{array} \right \}
\mbox{ in }\mc{Z}_T \\
\left. \begin{array}{l}
\phi_1(-\log T,\rho)= \rho (T^{\frac{p+1}{p-1}}g(T\rho)-\tfrac{2}{p-1}\kappa_p^{\frac{1}{p-1}} )\\
\phi_2(-\log T,\rho)=T^{\frac{2}{p-1}} ( T \rho f'(T\rho)+f(T\rho) )-  \kappa_p^{\frac{1}{p-1}}
\end{array} \right \}
\mbox{ for } \rho \in [0,1]
\end{array} \right .
\end{align} 
where $\mc{Z}_T:=\{(\tau,\rho): \tau> -\log T, \rho \in [0,1]\}$
and 
\begin{equation}
\label{eq:N}
N(x)=|\kappa_p^{\frac{1}{p-1}}+x |^{p-1}  (\kappa_p^{\frac{1}{p-1}}+x )-\kappa_p^{\frac{p}{p-1}} - p \kappa_p x.
\end{equation}
It is important to note that the original field can be reconstructed by 
\begin{align}
\label{Eq:ReconstructField}
\begin{split}
\psi(t,r) &= \psi^T(t,r) + (T-t)^{-\frac{2}{p-1}} r^{-1} \int_0^r \phi_2(-\log(T-t),\tfrac{r'}{T-t}) dr', \\
\psi_t(t,r)&= \psi_t^T(t,r) + (T-t)^{-\frac{2}{p-1}} r^{-1} \phi_1(-\log(T-t),\tfrac{r}{T-t}).
\end{split}
\end{align}
\section{Linear perturbation theory}

This section addresses the linearized problem, i.e., we drop the nonlinearity in \eqref{eq:nonlinear_firstorder_css}
and study the resulting equation as an abstract Cauchy problem in a suitable function space that will be introduced in the following. 

\subsection{The function space}

As usual, we endow $H^1(0,1)^2$ with the norm
\[ \|\bs u\|^2 := \|u_1\|^2_{H^1} +  \|u_2\|^2_{H^1}\]
where
\[\| u \|^2_{H^1} = \int_0^1 |u(\rho)|^2 d\rho + \int_0^1 |u'(\rho)|^2 d\rho\]
is induced by the standard inner product. 
We set $\mathcal H=\{u\in H^1(0,1): u(0)=0\}\times H^1(0,1)$. Now consider the sesquilinear form
$$(\bs u, \bs v)_{\bs 1} := (u_1(1) + u_2(1)) \overline{ (v_1(1) + v_2(1))} + 
\int_0^1 u_1'(\rho) \overline{v_1'(\rho)} d\rho + \int_0^1 u_2'(\rho) \overline{v_2'(\rho)}d\rho$$
and the associated quantity
$$ \|\bs u\|^2_{\bs 1} := (\bs u, \bs u)_{\bs 1} = |u_1(1) + u_2(1)|^2 + \|u_1'\|^2_{L^2} + \|u_2'\|^2_{L^2}.$$
\begin{lemma}\label{le:equiv_norms}
The quantity $\|\cdot \|_{\bs 1}$ defines a norm on $\mc H$ which is equivalent to  $\| \cdot\|$.
\end{lemma}
\begin{proof}
The map $\|\cdot \|_{\bs 1}$ indeed defines a norm on $\mc H$ since $\|\bs u\|_{\bs 1} = 0$ implies $u_1=c_1$, $u_2 = c_2$
for constants $c_1,c_2$ as well as $u_1(1)=-u_2(1)$. 
The boundary condition $u_1(0)=0$ shows that $c_1 = 0$ and thus, $c_2 =0$.

Next, we prove equivalence of the norms. Using the fact that $\|u_j\|_{L^{\infty}} \lesssim \|u_j\|_{H^1}$ for $j=1,2$ we immediately
obtain
\[ \| \bs u \|^2_{\bs 1}\lesssim |u_1(1)|^2 +|u_2(1)|^2 + \|u_1' \|_{L^2}^2 + \|u_2' \|_{L^2}^2 \lesssim \|u_1\|^2_{H^1} + \|u_2\|^2_{H^1} \lesssim \| \bs u \|^2. \]
In order to prove the reverse inequality we require estimates for the $L^2$-norms of the individual components. 
By using the fundamental theorem of calculus for absolutely continuous functions, the boundary condition for $u_1$, 
and the Cauchy-Schwarz inequality, we obtain
\[ |u_1(\rho)| \leq \int_0^{\rho} |u_1'(s)|ds \leq \|u_1'\|_{L^2}.\]
Squaring and integrating yields $\|u_1\|_{L^2} \leq \|u_1'\|_{L^2}$. To derive a similar estimate for $\|u_2\|_{L^2}$ we use the identity
\[ \int_{\rho}^1 u_j'(s) ds = u_j(1) - u_j(\rho) \]
for $j=1,2$ to infer that
\begin{align*}
 |u_1(\rho) + u_2(\rho)| & \leq  |u_1(1) + u_2(1)| + \int_{\rho}^1 |u_1'(s)| ds + \int_{\rho}^1 |u_2'(s)| ds  \\
 & \leq |u_1(1) + u_2(1)| +
\|u_1'\|_{L^2} +\|u_2'\|_{L^2} 
\end{align*}
by Cauchy-Schwarz. Hence,
\begin{align*}
|u_2(\rho)| & = |u_2(\rho)+u_1(\rho) -u_1(\rho)| \leq |u_2(\rho)+u_1(\rho)| + |u_1(\rho) | \\
 &\leq |u_1(1) + u_2(1)| + 2 \|u_1'\|_{L^2} +\|u_2'\|_{L^2} 
\end{align*}
where we used the above estimate for $u_1$. Squaring and integrating yields 
\begin{align*}
& \|u_2\|^2_{L^2} \lesssim  |u_1(1) + u_2(1)|^2 + \|u_1'\|^2_{L^2} + \|u_2'\|^2_{L^2} \lesssim \| \bs u \|^2_{\bs 1}.
\end{align*}
We conclude that
\begin{align*}
 \| \bs u\|^2 = \|u_1\|^2_{L^2} +  \|u_1'\|^2_{L^2} + \|u_2\|^2_{L^2} + \|u_2'\|^2_{L^2} \lesssim \| \bs u \|^2_{\bs 1}.
\end{align*}
\end{proof}

\subsection{Operator formulation -- well-posedness of the linearized problem}
In correspondence with the right-hand side of the linearization of Eq.~\eqref{eq:nonlinear_firstorder_css} we define the 
operators $(\bs{\tilde{L}}_0, \mc D(\bs{\tilde{L}}_0))$ and $\bs L' \in \mathcal{B}(\mathcal{H})$ by
\[\bs{\tilde{L}}_0\bs{u}(\rho):=\left ( \begin{array}{c}u_2'(\rho)-\rho u_1'(\rho)  \\ 
u_1'(\rho)-\rho u_2'(\rho)\end{array}  \right)-\tfrac{2}{p-1}\bs{u}(\rho)  \]
where $\mc D(\bs{\tilde{L}}_0):=\{\bs{u} \in C^2[0,1] \times C^2[0,1]: u_1(0)=0, u_2'(0) = 0\}$ and
\[ \bs L' \bs{u}(\rho):=\left ( \begin{array}{c} p\kappa_p \int_0^\rho u_2(s)ds 
\\ 0 \end{array} \right ).\] 
It follows by inspection that $\bs{\tilde L}_0$ has range in $\mc H$.
It is also immediate that $\bs{\tilde L}_0$ is densely defined in $\mc H$.
Furthermore, by exploiting the compactness of the embedding 
$H^1(0,1) \hookrightarrow L^2(0,1)$ it is easy to see that $\bs L'$ is a compact operator.
\begin{lemma}
\label{L_semigroup}
The operator $(\bs{\tilde{L}}_0, \mc D(\bs{\tilde{L}}_0))$ is closable and we denote its closure by 
$(\bs L_0, \mc D(\bs L_0))$. Consequently, 
\[\bs L := \bs L_0 + \bs L', \quad\mc D(\bs L) = \mc D(\bs L_0) \] 
is a well-defined closed linear operator and $\bs u \in \mc D(\bs L)$ implies that 
$u_j \in C[0,1] \cap C^1[0,1)$ for $j=1,2$ with the boundary conditions $u_1(0)=u_2'(0)=0$.

Furthermore, $\bs L$ is the generator of a strongly continuous one-parameter semigroup 
$\mb S: [0,\infty) \to \mathcal{B}(\mathcal{H})$
which satisfies
\begin{align}
\label{Est:Semigroup_L}
 \|\mb S(\tau) \bs u\| \leq  M e^{\omega \tau} \|\bs u\|
\end{align}
for all $\tau \geq 0$, a constant $M \geq 1$, and a $p$-dependent exponent $\omega > 0$.
\end{lemma}
\begin{proof}
We consider the Hilbert space $\mc H$ equipped with the norm $\| \cdot \|_{\bs 1}$. 
First, we show that $\bs{\tilde{L}}_0$ is a closable operator and its closure is the generator of a $C_0$-semigroup.
The next estimate is crucial for our approach.
By definition of $(\cdot|\cdot)_{\bs 1}$ we have
\begin{align*}
\mathrm{Re} (\bs{\tilde L}_0\bs u|\bs u)_{\bs 1}&=
\mathrm{Re}\int_0^1 [u_2''(\rho)-\rho u_1''(\rho)-u_1'(\rho)]\overline{u_1'(\rho)}d\rho \\
&\quad +\mathrm{Re}\int_0^1 [u_1''(\rho)-\rho u_2''(\rho)-u_2'(\rho)]\overline{u_2'(\rho)}d\rho
-\tfrac{2}{p-1}\|\bs u\|_{\bs 1}^2.
\end{align*}
Since $\mathrm{Re}(u'\overline{u})=\frac12 (|u|^2)'$, an integration by parts 
yields 
\[ \mathrm{Re} (\bs{\tilde L}_0\bs u|\bs u)_{\bs 1}\leq -\tfrac{2}{p-1}\|\bs u\|_{\bs 1}^2. \]

Next, we show that $\text{rg}(\lambda - \bs{\tilde{L}}_0)$ is dense for 
$\lambda := 1  - \tfrac{2}{p-1} > - \tfrac{2}{p-1}$. 
For arbitrary $\bs f =(f_1,f_2)\in \{(u_1,u_2)\in C^\infty[0,1]^2: u_1(0)=0\}$ (which is dense in $\mc H$) we set 
\[F(\rho) := f_1(\rho) + \rho f_2(\rho) + \int_0^{\rho} f_2(s) ds\] and define
\[u_1(\rho) := \rho u_2(\rho) - \int_0^{\rho} f_2(s) ds,  \quad u_2(\rho) := \frac{1}{1-\rho^2} \int_{\rho}^1 F(s) ds .\]
By Taylor's theorem it is immediate that $u_j \in C^2[0,1]$ for $j=1,2$ and 
we have $u_1(0) = 0$ as well as $u_2'(0) = -F(0) = 0$
which implies $\bs u =(u_1,u_2) \in \mc D(\bs{\tilde{L}}_0)$. 
A direct calculation shows that $(\lambda -\bs{\tilde{L}}_0) \bs u = \bs f$.
Consequently, the Lumer-Phillips Theorem (see \cite{EngNag00}, p.~83, Theorem 3.15) 
shows that $(\bs{\tilde{L}}_0, \mc D(\bs{\tilde{L}}_0))$ is closable and its closure $(\bs L_0, \mc D(\bs L_0))$ 
generates a strongly continuous 
one-parameter semigroup $\mb S_0: [0, \infty) \to \mc B(\mc H)$ which satisfies
\[ \| \mb S_0(\tau) \bs u \|_{\bs 1} \leq e^{-\frac{2}{p-1}\tau} \|\bs u\|_{\bs 1}\]
for $\bs u \in \mc H$.  Equivalence of the norms $\|\cdot \|_{\bs 1}$ and $\|\cdot \|$ on $\mc H$, 
which was shown in Lemma \ref{le:equiv_norms}, implies the existence of a constant $M  \geq 1$ such that 
\begin{align}\label{L0_semigroup}
 \| \mb S_0(\tau) \bs u \| \leq M e^{-\frac{2}{p-1}\tau} \|\bs u\|.
\end{align}

Next, we add the perturbation $\bs L' \in \mc B(\mc H)$ and set $\bs L:= \bs L_0 + \bs L'$. Boundedness of $\bs L'$ implies that 
$ \mc D(\bs L) = \mc D(\bs L_0)$. The Bounded Perturbation Theorem 
(see \cite{EngNag00}, p.~158) shows that $\bs L$ is the generator of a strongly continuous 
one-parameter semigroup $\mb S: [0, \infty) \to \mc B(\mc H)$ satisfying
\[ \| \mb S(\tau) \bs u \|  \leq M e^{ (M\|\bs L'\|-\frac{2}{p-1}) \tau} \|\bs u\|.\]

Finally, to characterize the generator in more detail assume that $\bs u \in \mc D(\bs L)=\mc D(\bs L_0)$. The fact that $u_j \in C[0,1]$ for $j=1,2$ and $u_1(0)=0$ follows immediately by Sobolev embedding since $\bs u \in \mc H$. By definition of the closure there exists a sequence 
$(\bs u_k) \subset \mc D(\bs{\tilde{L}}_0) \subset C^2[0,1] \times C^2[0,1]$ such that $\bs u_k \to \bs u$ and $\bs L_0 \bs u_k \to \bs L_0 \bs u$ in $\mc H$.
Sobolev embedding implies uniform convergence of the individual components and a suitable combination of the 
respective expressions shows that $(1-\cdot^2) u'_{1,k} \to (1-\cdot^2) u'_{1}$ and $(1-\cdot^2) u'_{2,k} \to (1-\cdot^2) u'_{2}$ uniformly.
We infer that for $j=1,2$,  $u'_{j,k}  \to u'_{j}$ in  $L^{\infty}(a,b)$ for any $(a,b] \subset (0,1)$ which shows that $u_j \in C^1[0,1)$.
As a consequence, $u'_{2,k}(\rho) \to u'_{2}(\rho)$ pointwise for $\rho \in [0,1)$ which yields the boundary condition $u_2'(0)=0.$
\end{proof}

\begin{corollary}\label{cor:wellposed_linearized}
The Cauchy problem \begin{equation*}
 \left \{ \begin{array}{l}
\frac{d}{d\tau}\Psi(\tau)=\bs L \Psi(\tau) \mbox{ for }\tau>0 \\
\Psi(0)=\bs{u} \in \mc D(\bs L)
          \end{array} \right .
\end{equation*}
has a unique solution $\Psi \in C^1([0,\infty), \mc H)$ which is given by
$ \Psi(\tau)=\mb S(\tau)\bs{u} $
for all $\tau \geq 0$.
\end{corollary}

\subsection{Spectral analysis of the generator}
The growth estimate for the semigroup $\mb S$ obtained in Lemma \ref{L_semigroup} by abstract results is not optimal. 
In order to refine \eqref{Est:Semigroup_L} we investigate the spectral properties of the generator.  
\begin{lemma}\label{spectrum}
We have $\sigma(\bs L) \subseteq \{\lambda \in \C: \mathrm{Re} \lambda \leq -\tfrac{2}{p-1}  \} \cup \{1\}$.
The spectral point $\lambda_{\bs g} =1$ is an eigenvalue and the associated one-dimensional geometric eigenspace is spanned by the symmetry mode
\begin{equation}\label{Eq:gaugemode}
\bs{g}(\rho):=\left ( \begin{array}{c}
\frac{p+1}{p-1} \rho  \\ 1 \end{array} \right ).
\end{equation}
\end{lemma}
\begin{proof}
We set $\mc{S}:= \{\lambda \in \C: \text{Re} \lambda \leq -\tfrac{2}{p-1}  \} \cup \{1\}$. Let $\lambda \in \sigma(\bs L)$. 
If $\text{Re} \lambda \leq -\tfrac{2}{p-1}$ then $\lambda \in \mc{S}$ trivially, hence assume that $\text{Re} \lambda > -\tfrac{2}{p-1}$. 
We show that under this assumption $\lambda \in \sigma_p(\bs L)$ and $\lambda = 1$.

From \eqref{L0_semigroup} and standard results
from semigroup theory (see \cite{EngNag00}, p.~55, Theorem 1.10) we infer that
\[ \sigma(\bs L_0) \subseteq \{\lambda \in \C: \text{Re} \lambda \leq -\tfrac{2}{p-1} \}.\] 
In particular, the above assumption on $\lambda$ implies that $\lambda \not \in \sigma(\bs L_0)$. 
We use the identity 
\[\lambda - \bs L = [1 - \bs L'\mb R_{\bs L_0}(\lambda)](\lambda - \bs L_0)\]
which shows that $1 \in \sigma(\bs L'\mb R_{\bs L_0}(\lambda))$, hence $1$ is an eigenvalue of the
compact operator $\bs L'\mb R_{\bs L_0}(\lambda)$. Let $\bs f \in \mc H$ denote the corresponding eigenvector.
Setting $\bs u := \mb R_{\bs L_0}(\lambda) \bs f$ yields $\bs u \in \mc D(\bs L_0)=\mc D(\bs L)$,
$\bs u\not= \bs 0$,
 as well as $(\lambda - \bs L) \bs u = \bs 0$, and we conclude that $\lambda \in \sigma_p(\bs L)$. 

The eigenvalue equation $(\lambda - \bs L)\bs u = \mb 0$ implies that (see \cite{DonSch12})
\begin{align}
\label{Eq:u1}
u_1(\rho)= \rho u_2(\rho)+ (\lambda  + \tfrac{3-p}{p-1} )\int_0^\rho u_2(s)ds, 
\end{align}
as well as
\begin{align} 
\label{Eq:Eigenval}
(1-\rho^2)  &u''(\rho)-\left(2\lambda + \tfrac{4}{p-1} \right) \rho u'(\rho) \nonumber \\
&- \left [ \left(\lambda+ \tfrac{2}{p-1} \right)\left(\lambda + \tfrac{3-p}{p-1}\right) - p\kappa_p \right ] u(\rho) = 0
\end{align}
where $u(\rho):= \int_0^{\rho} u_2(s) ds$. Since $u_2 \in H^1(0,1)$ for $\bs u \in \mc H$ we have $u \in  H^2(0,1)$. 
Furthermore,
$\bs u \in \mc D(\bs L)$ yields $u \in C^2[0,1)$ and the boundary conditions $u(0)=u''(0)=0$, 
see Lemma \ref{L_semigroup}.
We substitute 
$\rho \mapsto z:= \rho^2$ to obtain the hypergeometric differential equation
\begin{equation}
\label{eq_hypgeom}
z(1-z)v''(z)+[c-(a+b+1)z]v'(z)-abv(z)=0 
\end{equation} 
where $v(z):=u(\sqrt{z})$ and the parameters are given by 
\begin{align*}
a= \tfrac12 (\lambda -2), \quad b=\tfrac12 (\lambda + \tfrac{p+3}{p-1}), \quad c=\tfrac12.
\end{align*}
 At $z=1$ the exponents of the indicial equation 
are $\{0, c-a-b\}$, where $c-a-b = \frac{p-3}{p-1} -\lambda$. 
The assumption $\mathrm{Re}\lambda>-\frac{2}{p-1}$ implies $\mathrm{Re}(c-a-b)<1$ and thus,
by Frobenius' method it follows that there exist two linearly independent solution
$v_1$ and $\tilde v_1$ of Eq.~\eqref{eq_hypgeom} with the asymptotic behavior
\[ v_1(z)\sim 1,\quad \tilde v_1(z)\sim (1-z)^{c-a-b} \mbox{ as }z\to 1-, \]
at least if $c-a-b\not =0$.
In the degenerate case $c-a-b=0$ we have $\tilde v_1(z)\sim \log(1-z)$ as $z\to 1-$.
In fact, $v_1$ is given explicitly by
\[ 
v_1(z) = {}_2F_1(a, b; a+b+1-c;1-z) \]
where ${}_2F_1$ denotes the standard
hypergeometric function, see e.g.~\cite{DLMF}.  
The assumption $\text{Re} \lambda > - \frac{2}{p-1}$ implies that $v =\alpha v_1$ for some 
constant $\alpha \in \C$ because otherwise the corresponding $u(\rho)=v(\rho^2)$ would not belong to $H^2(0,1)$. 
Another fundamental system $\{v_0,\tilde v_0\}$ of Eq.~\eqref{eq_hypgeom} is given by
\begin{align*}
\tilde{v}_0(z)& = {}_2F_1(a,b;c;z), \\
v_0(z)& = z^{1/2}{}_2F_1(a+1-c,b+1-c; 2-c; z),
\end{align*}
see \cite{DLMF}, and there must exist constants
$c_0$, $c_1 \in \C$ such that 
\[v_1=c_0 \tilde{v}_0+c_1 v_0. \]
The connection coefficients $c_0$ and $c_1$ are known explicitly in terms of the $\Gamma$-function,
see \cite{DLMF}.
The condition $u(0)=0$ implies that $v(0)=v_1(0)= 0$ and thus,
\[c_0=\frac{\Gamma(a+b+1-c)\Gamma(1-c)}{\Gamma(a+1-c)\Gamma(b+1-c)}\]
must vanish. This can only be the case when at least one of the Gamma functions in the denominator has a pole, 
which is equivalent to
\begin{align*}
\tfrac12(\lambda -1)=-k \quad \mbox{or} \quad  \tfrac{\lambda}{2} + \tfrac{p+1}{p-1}=-k \quad \text{for} \quad k \in \N_0.
\end{align*}
The latter condition can be rewritten as
$\lambda = -2k -\frac{2p}{p-1}-\frac{2}{p-1}$ 
which implies that $\lambda < -\frac{2}{p-1}$ but this is excluded by assumption. 
The first condition is satisfied if
$\lambda = 1 - 2k \in \{1,-1,-3, \cdots\}$ and since $-\frac{2}{p-1} \in (-1,0)$, 
we see that $\lambda = 1$ is the only possibility. We denote this particular eigenvalue by $\lambda_{\bs g}$.
For $\lambda=\lambda_{\bs g}=1$ we have
$v_1(z) = c_1 \sqrt{z}$ and $u(\rho) = \alpha \rho$ for some $\alpha \in \C$. 
In particular, $u$ satisfies the boundary conditions $u(0)=u''(0)=0$. 
Finally, from Eq.~\eqref{Eq:u1} we obtain 
$u_1(\rho)= \alpha \frac{p+1}{p-1} \rho$, $u_2(\rho)=\alpha$ which shows that the geometric eigenspace
associated to $\lambda_{\bs g}$ is spanned by $\bs g$ as claimed.
\end{proof}

\subsection{Resolvent bounds}

\begin{lemma}
\label{Le:resolvent}
Fix $\varepsilon > 0$. Then there exist constants $c_1,c_2 > 0$ 
such that 
\[ \|\mb R_{\bs L}(\lambda)\| \leq c_2 \]
for all $\lambda \in \C$ with $\mathrm{Re} \lambda \geq -\tfrac{2}{p-1} + \varepsilon$ and $|\lambda| \geq c_1$. 
\end{lemma}
\begin{proof}
In view of the identity 
\[\mb R_{\bs L}(\lambda)=\mb R_{\bs L_0}(\lambda)[1-\bs{L'}\mb{R}_{\bs L_0}(\lambda)]^{-1} \]
it suffices to prove smallness of $\|\bs{L'}\mb{R}_{\bs L_0}(\lambda)\|$.
First note that semigroup theory yields  (see \cite{EngNag00}, p.~55, Theorem 1.10)
\begin{align}
\label{Eq:RL0_bounds}
\| [\mb R_{\bs L_0}(\lambda) \bs f]_j\|_{H^1} \leq \| \mb R_{\bs L_0}(\lambda)\bs f\| \leq  \frac{M \|\bs f\|}{\text{Re} \lambda + \frac{2}{p-1}} 
\end{align}
for $j=1,2$, $\bs f \in \mc H$, and $M$ is the constant from Lemma \ref{L_semigroup}. Suppose we have $\|\bs L' \mb R_{\bs L_0}(\lambda)\| \leq c <1$ for  $c >0$ 
and $|\lambda| \geq c_1$ large enough. Then this implies 
\[ \|\mb R_{\bs L}(\lambda)\| \leq \|\mb R_{\bs L_0}(\lambda)\| (1-\|\bs{L'}\mb{R}_{\bs L_0}(\lambda)\|)^{-1} \leq  c_2 \]
where $c_2 \to \infty$ as $\varepsilon \to 0+$.
Note that
\[\bs{L'}\mb{R}_{\bs L_0}(\lambda) \bs f =  \left ( \begin{array}{c} p\kappa_p V[\mb R_{\bs L_0}(\lambda)\bs f]_2
\\ 0 \end{array} \right )\]
where $V: H^1(0,1) \to \{u \in H^1(0,1): u(0)=0\}$ is a bounded operator defined by $Vu(\rho):=\int_0^{\rho} u(s) ds$. 
For all $\bs f\in \mc H$ we have $(\lambda - \bs L_0)\mb R_{\bs L_0}(\lambda) \bs f = \bs f$ which implies
\[[\mb R_{\bs L_0}(\lambda)\bs f]_1(\rho) = (\lambda - \tfrac{p-3}{p-1}) V[\mb R_{\bs L_0}(\lambda)\bs f]_2(\rho) +
\rho [\mb R_{\bs L_0}(\lambda)\bs f]_2(\rho) - Vf_2(\rho)\]
and this yields the estimate
\begin{align*}
|\lambda - \tfrac{p-3}{p-1}| \| V[\mb R_{\bs L_0}(\lambda)\bs f]_2 \|_{H^1} \lesssim  
\| [\mb R_{\bs L_0}(\lambda)\bs f]_1\|_{H^1}  + \| [\mb R_{\bs L_0}(\lambda)\bs f]_2 \|_{H^1} + \|f_2\|_{H^1}. 
\end{align*}
Using \eqref{Eq:RL0_bounds} we obtain
\[\|\bs L' \mb R_{\bs L_0}(\lambda)\bs f\| = p\kappa_p\| V[\mb R_{\bs L_0}(\lambda)\bs f]_2 \|_{H^1} 
\lesssim \frac{ \|\bs f\| }{|\lambda - \tfrac{p-3}{p-1}|} \]
such that $\|\bs L' \mb R_{\bs L_0}(\lambda)\| \leq \frac12$ for all $|\lambda|$ sufficiently large.
\end{proof}

\subsection{A growth estimate for the linearized evolution}
\begin{lemma}
\label{Le:LinearTimeEvol}
Let  $\varepsilon>0$ be fixed and so small that
\[\mu_p := \tfrac{2}{p-1} - \varepsilon > 0.\]
Then there exists a projection $\bs P \in \mc B(\mc H)$ onto  $\langle \bs g \rangle$ 
which commutes with the semigroup $\mb S(\tau)$ for all $\tau \geq 0$ and
\begin{align*}
\|\mb S(\tau)(1-\bs P)\bs f\| &\leq C_{\varepsilon} e^{-\mu_p \tau}\|(1-\bs P)\bs f\| \\
\mb S(\tau) \bs P \bs f &= e^{\tau} \bs P \bs f
\end{align*}
for all $\tau \geq 0$, $\bs f \in \mc H$, and a constant $C_{\varepsilon} > 0$.
\end{lemma}

\begin{proof}
Let $\gamma$ be a (positively oriented) circle around $\lambda_{\bs g}$ with 
radius $r_{\gamma} = \frac12$. By Lemma \ref{spectrum}, $\gamma$ belongs to the resolvent set 
of $\bs L$ and no spectral points lie inside of $\gamma$ except for $\lambda_{\bs g}$. According to \cite{Kat95}, p.~178, 
Theorem 6.5, a spectral projection  $\bs P \in \mc B(\mc H)$ is defined by
\begin{align*}
\label{Eq:Spectral_proj}
\bs P=\frac{1}{2\pi i}\int_\gamma \mb R_{\bs {L}}(\lambda) d\lambda,
\end{align*}
where $\bs P$ commutes with $\bs L$ in the sense that $\bs P \bs L \subset \bs L \bs P$.
Furthermore, $\bs P$ commutes with the resolvent of $\bs L$, see \cite{Kat95} p.~173, Theorem 6.5. 
This implies that $\bs P$ commutes with the linear time evolution, i.e., 
$\bs P \mb S(\tau) = \mb S(\tau) \bs P$ for $\tau \geq 0$, where $\mb S: [0,\infty) \to \mc B(\mc H)$ is the semigroup generated by $\bs L$.

Most important is that $\bs L$ is decomposed according to the decomposition of the Hilbert space 
$\mc H = \mathrm{ker} \bs P \oplus\mathrm{rg} \bs P $ into parts 
$\bs L \rst{{\mathrm{ker} \bs P}}$ and $\bs L \rst{{\mathrm{rg} \bs P}}$,  where 
$\mc D(\bs L \rst{{\mathrm{ker} \bs P}}) = \mc D(\bs L) \cap \ker \bs P$ and 
$\bs L \rst{ {\mathrm{ker} \bs P}} \, \bs u =  \bs L \bs u$ for 
$\bs u \in \mc D(\bs L \rst{ {\mathrm{ker} \bs P}})$ (an analogous definition holds for 
$\bs L \rst{{ \mathrm{rg} \bs P}}$).
Moreover,
\[\sigma(\bs L \rst{ {\mathrm{ker} \bs P}}) = \sigma(\bs L) \setminus \{1\}, 
\quad \sigma(\bs L \rst{{ \mathrm{rg} \bs P}}) = \{1\}.\]
Since $\bs L$ is not self-adjoint, we only know a priori that 
$\ker(\lambda_{\bs g} - \bs L)=\langle \bs g \rangle \subseteq \mathrm{rg} \bs P$ and it remains to show that 
$\mathrm{rg} \bs P = \langle \bs g \rangle$. This is equivalent to the fact that the algebraic 
multiplicity of $\lambda_{\bs g}$ 
is equal to one and for this we refer to \cite{DonSch12}, where the proof of Lemma $3.7$ can be 
copied verbatim.

Having this, it is easy to see that $\mb S(\tau) \bs P \bs f = e^{\tau} \bs P \bs f$  for $\bs f \in \mc H$. 
In order to obtain an estimate on the stable subspace, we use the structure of the spectrum of $\bs L\rst{{\mathrm{ker} \bs P}}$ and  
Lemma \ref{Le:resolvent} which imply that for all $\lambda \in \C$ with $\mathrm{Re} \lambda \geq -\frac{2}{p-1} + \varepsilon$,
the restriction of the resolvent $\mb R_{\bs L}(\lambda)$ to $\mathrm{ker} \bs P$ (which equals the resolvent 
of $\bs L\rst{{\mathrm{ker} \bs P}}$) exists and is uniformly bounded. 
Since $\mathrm{ker}\bs P$ is a Hilbert space, we can apply the theorem by Gearhart, Pr\"{u}ss, and Greiner 
(see e.g.~\cite{EngNag00}, p.~302, Theorem 1.11 or \cite{HelSjo10}) to obtain the claimed estimate.
\end{proof}

\section{Nonlinear Perturbation Theory}

The aim of this section is to prove the existence of solutions of 
the full nonlinear equation \eqref{eq:nonlinear_firstorder_css} which retain
the exponential decay of the linearized problem on the stable subspace, cf. Lemma \ref{Le:LinearTimeEvol}.

Note that the exponential growth of the semigroup on the unstable subspace 
$\bs P \mc H$ has its origin in the time translation symmetry of the original equation, 
which will become clear in the following. In fact,
we are perturbing around a one-parameter family of solutions and it  
is clear that a generic perturbation around $\psi^{T^*}$
for a particular fixed value $T^*$ will change the blow up time.
Therefore, we expect the solution to converge to  $\psi^{T}$ where
in general $T^* \neq T$.

Without loss of generality we set $T^*=1$ and study perturbations
around $\psi^1$. The blow up time $T$ will be considered
as a variable that will be fixed later on in the proof. 

Note that from now on we restrict ourselves to real-valued functions.

\subsection{Estimates for the nonlinearity}
We formally set 
\[Ku(\rho):=\frac{1}{\rho} \int_0^{\rho} u(s) ds.\]
\begin{lemma}\label{le:Integral_Linfty}
Let $u \in H^1(0,1)$. Then $Ku \in L^{\infty}(0,1)$ and there exists a $c > 0$ such that
$\|Ku \|_{\infty} \leq c \|u\|_{H^1}$.
\end{lemma}
\begin{proof}
By the continuous embedding $H^1(0,1)\hookrightarrow L^\infty(0,1)$ we see that $u\in L^\infty(0,1)$.
In particular,
\[ |Ku(\rho)| \leq \frac{1}{\rho} \int_0^{\rho} |u(s)| ds \leq \|u \|_{\infty} \lesssim \|u\|_{H^1} \]
for all $\rho \in [0,1]$.
\end{proof}

In order to define the nonlinearity we introduce the auxiliary function 
$N: \R \times [0,1]\to\R$ given by
\begin{align}\label{Eq:Nonlin_RealVal}
N(x,\rho):= \rho \bigg[ |\kappa_p^{\frac{1}{p-1}}+x|^{p-1} (\kappa_p^{\frac{1}{p-1}}+x) - p\kappa_p x - 
\kappa_p^{\frac{p}{p-1}} \bigg ],
\end{align}
cf.~Eq.~\eqref{eq:N}. Since $p >  3$, the function $N$ is at least twice continuously differentiable on $\R$ with respect 
to the first variable. Furthermore, for any fixed $\rho\in [0,1]$, we have
$N(x,\rho) = O(x^2)$ as $x \rightarrow 0$. 
Thus, it is easy to see that with  $\langle x \rangle:=\sqrt{1+|x|^2}$,
\begin{align} \label{est:nonlin0}
|N(x,\rho)| &\lesssim \rho |x|^2\langle x\rangle^{p-2} \quad & |\partial_1 N(x,\rho)| &\lesssim   \rho |x|\langle x\rangle^{p-2} \nonumber \\
|\partial^2_1 N(x,\rho)|  &\lesssim \rho \langle x\rangle^{p-2} & \quad  |\partial_2 N(x,\rho)|  &\lesssim   |x|^2\langle x\rangle^{p-2}
\end{align}
for all $x\in\R$ and $\rho \in [0,1]$.

We formally define a vector-valued nonlinearity by 
 \[ \bs N(\bs u) (\rho) := \left ( \begin{array}{c} N(Ku_2(\rho),\rho)
\\  0 \end{array} \right ).\] 
With this definition, Eq.~\eqref{eq:nonlinear_firstorder_css} can be (formally) written as an ordinary
differential equation for a function $\Phi: [-\log T,\infty)\to \mc H$ given by
\begin{equation}
\label{eq:nlop} \tfrac{d}{d\tau}\Phi(\tau)=\bs L \Phi(\tau)+\bs N(\Phi(\tau)),\quad \tau>-\log T 
\end{equation}
with initial data
\begin{equation}
\label{eq:initialdata}
 \Phi(-\log T)(\rho) =\left ( \begin{array}{c} \rho T^{\frac{p+1}{p-1}}g(T\rho)-\tfrac{2 \rho}{p-1}\kappa_p^{\frac{1}{p-1}} 
\\ T^{\frac{2}{p-1}}\left (T \rho f'(T\rho)+f(T\rho) \right ) -  \kappa_p^{\frac{1}{p-1}} \end{array} \right ).
\end{equation}
In the following we denote by $\mc B_1$ the open unit ball in $(\mc H,\|\cdot \|)$.
\begin{lemma}\label{lemma:nonlinL2}
The operator $\bs N$ maps $\mc H$ to $\mc H$ and there exists a constant $c > 0$ such that
\[\|\bs N(\bs u)-\bs N(\bs v)\|\leq c (\| \bs u \| + \| \bs v \|) \|\bs u-\bs v\|\]
for all $\bs u,\bs v \in \mc B_1$. 
Furthermore, $\bs N(\bs 0) = \bs 0$ and $\bs N$ is Fr\'{e}chet differentiable at $\mb 0$ with
$D\mb {N}(\mb 0)=\mb 0$.
\end{lemma}

\begin{proof}
First, we derive some estimates for the real-valued function $N$ defined in \eqref{Eq:Nonlin_RealVal}. 
We use the fundamental theorem of calculus and \eqref{est:nonlin0} to obtain
\begin{align} \label{est:nonlin1}
\begin{split}
| \partial_1 N(x,\rho) - \partial_1 N(y,\rho)|  & \leq |x-y| \int_0^1 |\partial^2_1 N(y+h(x-y),\rho)| dh  \\
& \lesssim \rho  |x-y|  \int_0^1   \langle y+h(x-y) \rangle^{p-2} dh  \\
&\lesssim \rho  |x-y| [\langle x \rangle^{p-2} +\langle y \rangle^{p-2}]
\end{split}
\end{align}
for $x,y \in \R$. Similarly, 
\begin{align} \label{est:nonlin2}
\begin{split}
 |  N(x,\rho) -  N(y,\rho)| & \leq |x-y| \int_0^1 |\partial_1 N(y+h(x-y),\rho)| dh \\
& \lesssim \rho  |x-y| [|x| \langle x \rangle^{p-2} +|y| \langle y \rangle^{p-2}].
\end{split}
\end{align}
Note that $\|\bs N(\bs u)-\bs N(\bs v)\| = \|[\bs N(\bs u)]_1 - [\bs N(\bs v)]_1 \|_{H^1}$ and we obtain
\begin{align*}
\|[\bs N(\bs u)]_1 - [\bs N(\bs v)]_1 \|_{H^1}^2 & = \int_0^1 |N(Ku_2(\rho),\rho) - N(Kv_2(\rho),\rho)|^2 d\rho  \\
&\quad + \int_0^1 \left |\tfrac{d}{d\rho}[N(Ku_2(\rho),\rho) - N(Kv_2(\rho),\rho)] \right|^2  d\rho. 
\end{align*}
With Lemma \ref{le:Integral_Linfty} and \eqref{est:nonlin2} we get
\begin{align*}
 I_0:&= \int_0^1 |N(Ku_2(\rho),\rho) - N(Kv_2(\rho),\rho)|^2 d\rho \\
  &\lesssim \int_0^1 \rho^2 |Ku_2(\rho) - Kv_2(\rho)|^2 
  \Big[|Ku_2(\rho)|^2 \langle Ku_2(\rho) \rangle^{2(p-2)} \\
  &\quad +|Kv_2(\rho)|^2 \langle Kv_2(\rho) \rangle^{2(p-2)}  \Big ] d\rho \\
  &\lesssim \left[\|Ku_2\|_{\infty}^2 \langle \|Ku_2\|_{\infty} \rangle^{2(p-2)} +\|Kv_2\|_{\infty}^2 
  \langle \|Kv_2\|_{\infty} \rangle^{2(p-2)} \right] \\
  &\quad \times \|K(u_2-v_2)\|^2_{\infty} \\
  &\lesssim  \left[\|u_2\|_{H^1}^2 \langle \|u_2\|_{H^1} \rangle^{2(p-2)} +\|v_2\|_{H^1}^2  \langle\|v_2\|_{H^1} \rangle^{2(p-2)} \right] \|u_2 - v_2\|_{H^1}^2.
\end{align*}
For the second term we obtain
\begin{align*}
\int_0^1 &\left |\tfrac{d}{d\rho}[N(Ku_2(\rho),\rho) - N(Kv_2(\rho),\rho)] \right|^2  d\rho  \\
 &\lesssim  \int_0^1 \left |\partial_1 N(Ku_2(\rho),\rho)(Ku_2)'(\rho) - \partial_1 N(Kv_2(\rho),\rho)  (Kv_2)'(\rho) \right|^2  d\rho  \\
 &\quad + \int_0^1  |\partial_2 N( Ku_2(\rho), \rho) -\partial_2 N( Kv_2(\rho), \rho)|^2 d\rho  \\
 &\lesssim I_1 + I_2 + I_3, 
\end{align*}
where
\begin{align*}
I_1 & :=\int_0^1  | (Ku_2)'(\rho) |^2 |\partial_1 N( Ku_2(\rho), \rho) -\partial_1 N( Kv_2(\rho), \rho)|^2 d\rho \\
I_2 & :=\int_0^1  |\partial_1 N(Kv_2(\rho),\rho)|^2 | (Ku_2)'(\rho)  -  (Kv_2)'(\rho)|^2 d\rho  \\
I_3 & :=\int_0^1  |\partial_2 N( Ku_2(\rho), \rho) -\partial_2 N( Kv_2(\rho), \rho)|^2 d\rho. 
\end{align*}
Estimate \eqref{est:nonlin1} and Lemma \ref{le:Integral_Linfty} yield
\begin{align*}
I_1 & \lesssim \int_0^1 \left |u_2(\rho) - Ku_2(\rho) \right |^2 | Ku_2(\rho) - Kv_2(\rho)|^2 \\
&\quad \times [\langle Ku_2(\rho) \rangle^{2(p-2)}+\langle Kv_2(\rho)\rangle^{2(p-2)}] d\rho \\
& \lesssim \left[\langle \|Ku_2\|_{\infty} \rangle^{2(p-2)} + \langle \|Kv_2\|_{\infty} \rangle^{2(p-2)} \right] \|u_2\|_{H^1}^2
\|K(u_2-v_2)\|^2_{\infty}  \\
& \lesssim 
\left[ \langle \|u_2\|_{H^1} \rangle^{2(p-2)} +\langle\|v_2\|_{H^1}  \rangle^{2(p-2)} \right] \|u_2 \|_{H^1}^2  \|u_2 - v_2\|_{H^1}^2.
\end{align*}
With estimate \eqref{est:nonlin0} we obtain
\begin{align*}
I_2 & \lesssim  \int_0^1 |Kv_2(\rho)|^2 \langle Kv_2(\rho)\rangle^{2(p-2)} \left[| u_2(\rho) -v_2(\rho)|^2 +|K(u_2-v_2)(\rho)|^2\right] d\rho  \\
& \lesssim \|Kv_2\|_{\infty}^2 \langle \|Kv_2\|_{\infty} \rangle^{2(p-2)} [\|u_2 -v_2\|_{L^2}^2 +\|K(u_2-v_2)\|^2_{\infty} ]\\
& \lesssim \|v_2\|_{H^1}^2 \langle\|v_2\|_{H^1} \rangle^{2(p-2)}\|u_2 -v_2\|_{H^1}^2.
\end{align*}
Since $\partial_2 N(x,\rho) = \rho^{-1} N(x,\rho)$, the third term can be estimated using \eqref{est:nonlin2} 
\begin{align*}
I_3  & = \int_0^1 \rho^{-2} | N(Ku_2(\rho),\rho) - N(Kv_2(\rho),\rho)|^2 d\rho  \\
& \lesssim  \left[\|u_2\|_{H^1}^2 \langle \|u_2\|_{H^1} \rangle^{2(p-2)} +\|v_2\|_{H^1}^2  \langle\|v_2\|_{H^1}  \rangle^{2(p-2)} \right] \|u_2 - v_2\|_{H^1}^2.
\end{align*}
Summing up yields
\begin{align*}
  I_0 + I_1 + I_2 + I_3 & \lesssim  \big [\|u_2\|_{H^1}^2 \langle \|u_2\|_{H^1}\rangle^{2(p-2)}  +\|v_2\|_{H^1}^2  \langle\|v_2\|_{H^1}  \rangle^{2(p-2)}   \\
  & \quad+ \|u_2\|_{H^1}^2 \langle \|v_2\|_{H^1} \rangle^{2(p-2)} \big] \|u_2 - v_2\|_{H^1}^2.
\end{align*}
In particular, for $\bs u \in \mc B_1$ we have $\|u_2\|_{H_1} \leq 1$ and thus $\langle\|u_2\|_{H^1}  \rangle \lesssim 1$. This yields 
\begin{align*}
\|[\bs N(\bs u)]_1 - [\bs N(\bs v)]_1 \|_{H^1}^2 & \lesssim \left (
\|u_2\|_{H^1}^2  +\|v_2\|_{H^1}^2 \right ) \|u_2 - v_2\|_{H^1}^2 \\
& \lesssim \left (\|\bs u\|^2  +\|\bs v\|^2 \right ) \|\bs u - \bs v \|^2  
\end{align*}
for $\bs u, \bs v \in \mc B_1$ and we conclude that 
\[\|\bs N(\bs u)-\bs N(\bs v)\| \lesssim \left(\|\bs u\|  +\|\bs v\| \right ) \|\bs u - \bs v \|. \]
The fact that $N(0,\rho)  = 0$ for all $\rho \in [0,1]$ yields $\bs N(\mb 0) = \mb 0$ such that the above estimate implies
$\|\bs N(\bs u) \| \lesssim  \|\bs u \|^2$.
In particular, \[ \frac{\|\bs N(\bs u) \|}{ \|\bs u \|} \rightarrow 0\] for $\bs u \rightarrow \mb 0$
which proves that $\bs N$ is differentiable at zero with $D\bs N(\mb 0) = \mb 0$.
 \end{proof}
\subsection{Abstract formulation of the nonlinear problem}
 
Next, we rewrite the initial data Eq.~\eqref{eq:initialdata} by setting
\[\bs U(\bs v,T)(\rho):=T^{\frac{2}{p-1}} [\bs v(T \rho) + \bs \kappa(T\rho)] - \bs{\kappa}(\rho)\]
where
\begin{align*}
\bs v(\rho):= \left ( \begin{array}{c}   \rho g(\rho)  \\
		        \rho f'(\rho)+f(\rho)         \end{array} \right )-\bs \kappa(\rho), \quad 
			 \bs \kappa(\rho):= \kappa_p^{\frac{1}{p-1}} \left ( \begin{array}{c} \frac{2 \rho}{p-1}  \\ 1 \end{array} \right ).
\end{align*}	
Eq.~\eqref{eq:initialdata} is equivalent to $\Phi(-\log T)=\bs U(\bs v,T)$.
The point is that $\bs v$ denotes the data relative to $\psi^1$ such that we have clearly 
separated the functional dependence of the initial data on the
free functions $(f,g)$ (or $\bs v$, respectively) and the blow up time $T$. In the following we are interested in mild solutions of the equation
\begin{align}
\label{Eq:Abstract_NonlinearDiff}
\left \{ \begin{array}{l}
\frac{d}{d\tau}\Psi(\tau)=\bs L\Psi(\tau)  + \bs N(\Psi(\tau)) \mbox{ for }\tau>0 \\
\Psi(0)=\bs U(\bs v,T)
\end{array} \right .
\end{align}
such that a solution of Eq.~\eqref{eq:nlop}
for a particular $T > 0$ can be obtained by setting $\Phi(\tau) := \Psi(\tau + \log T)$.
Note that in order to obtain a well-posed initial value problem, the initial data $(f,g)$
have to be defined on the spatial interval $[0,T]$. Since we do not know the blow up
time in advance we restrict $T$ to the interval $\mc I:= (\frac12, \frac32)$, which 
is no limitation since our argument is perturbative around $T=1$ anyway.
With these preliminaries we can rigorously define the initial data as a function 
of the free data $\bs v$ and the blow up time $T$ on $\mathfrak H \times \mc I$ where
 \[{\mathfrak H}:= \{u \in H^1(0,\tfrac32): u(0) = 0\} \times H^1(0,\tfrac32).\]
\begin{lemma}\label{initialdata_op}
The function $\bs U: \mathfrak H \times \mc I \to \mc H$ is continuous and $\bs U(\bs 0,1) = \bs 0$. 
Furthermore $\bs U(\bs 0,\cdot):  \mc I \to \mc H$ is Fr\'echet differentiable and
 \[[D_T \bs U(\bs 0, T)\rst{_{T=1}} \lambda](\rho) = \tfrac{2\lambda}{p-1} \kappa_p^{\frac{1}{p-1}} \bs g(\rho),\]
for $\lambda \in \R$ where $\bs g$ denotes the symmetry mode.
\end{lemma}
\begin{proof}
We first consider the function $M: H^1(0,\tfrac32) \times \mc I \to H^1(0,1)$ defined by $M(v,T)(\rho):=v(T\rho)$ and show that it is 
continuous. The estimate
\begin{align*}
\|M&(v,T) -  M(\tilde v, T) \|^2_{H^1(0,1)}  \\
& =  \int_0^1 | v(T\rho) - \tilde  v(T\rho)|^2 d\rho + T^2 \int_0^1| v'(T\rho) 
- \tilde  v'(T\rho)|^2 d\rho  \\
& =  \frac{1}{T} \int_0^T | v(\rho) - \tilde  v(\rho)|^2 d\rho 
+ T \int_0^T | v'(\rho) - \tilde  v'(\rho)|^2 d\rho  \\
 & \leq 2 \|v - \tilde v \|_{H^1(0,\frac32)}^2
\end{align*}
implies that $M(v,T)$ is continuous with respect to $v$, uniformly in $T \in \mc I$. 
Hence, it is sufficient to prove continuity with respect to $T$. 
For any $v, \tilde v \in H^1(0,\frac32)$ and $T, \tilde T \in \mc I$ we have 
\begin{align*}
\|M(v,T) -  M(v,\tilde T) \|_{H^1(0,1)} &\leq \|M(v,T) -  M(\tilde v, T) \|_{H^1(0,1)} \\
&\quad + \|M(\tilde v,T) -  M(\tilde v, \tilde T) \|_{H^1(0,1)} \\
&\quad  +\|M(\tilde v,\tilde T) -  M(v, \tilde T) \|_{H^1(0,1)}  \\
&  \lesssim \|v - \tilde v \|_{H^1(0,\frac32)} + \|M(\tilde v,T) -  M(\tilde v, \tilde T) \|_{H^1(0,1)} 
\end{align*}
We use the density of $C^1[0,\frac32]$ in $H^1(0,\frac32)$ to infer that for any given $\varepsilon > 0$
there exists a $\tilde v \in C^1[0,\frac32]$ such that 
\begin{align*}
\|M(v,T) -  M(v,\tilde T) \|_{H^1(0,1)}^2 < &\frac{\varepsilon^2}{2}  
+ C \int_0^1 | \tilde v(T\rho) - \tilde  v(\tilde T\rho)|^2 d\rho \\
& + C\int_0^1| T \tilde v'(T\rho) 
- \tilde T \tilde  v'(\tilde T\rho)|^2 d\rho 
\end{align*}
and the integral terms tend to zero in the limit $\tilde T \to T$ by continuity of $\tilde v$ and  $\tilde v'$. This implies the
claimed continuity of $M$ on $H^1(0,\tfrac32) \times \mc I$. Thus, for $\bs v = (v_1,v_2) \in \mathfrak H$, $T \in \mc I$ and 
$\bs \kappa = (\kappa_1,\kappa_2)$ as defined above, the function $\bs U$ can be written as
\[\bs U(\bs v,T)=  \left ( \begin{array}{c} T^{\frac{2}{p-1}}[ M(v_1,T)+ M(\kappa_1,T) ] - 
\kappa_1 \\ T^{\frac{2}{p-1}}[ M(v_2,T)+ M(\kappa_2,T) ] - \kappa_2 \end{array} \right ).\]
The properties of $M$ imply that $[\bs U(\bs v,T)]_j \in H^1(0,1)$ for $j=1,2$. Furthermore, we have 
$[\bs U(\bs v,T)]_1(0) = 0$ and $\bs U$ depends continuously on $(\bs v, T)$. 

Evaluation yields
\[ \bs U(\bs 0,T)(\rho)= \kappa_p^{\frac{1}{p-1}}  \left ( \begin{array}{c} \frac{2 \rho}{p-1} \left [T^{\frac{p+1}{p-1}} - 1 \right]\\ T^{\frac{2}{p-1}} - 1 \end{array} \right )\]
and obviously, 
$\bs U(\bs 0, \cdot): \mc I \to \mc H$ is differentiable for all $T \in I$. In particular, we have 
\[ [D_T \bs U(\bs 0, T)\rst{_{T=1}} \lambda](\rho) = \frac{2\lambda}{p-1}\kappa_p^{\frac{1}{p-1}} 
 \left ( \begin{array}{c} \frac{p+1}{p-1} \rho \\ 1 \end{array} \right ) \]
 which concludes the proof.
\end{proof}
 
Since we interested in mild solutions of \eqref{Eq:Abstract_NonlinearDiff}, we use Duhamel's formula to obtain
\begin{align}\label{eq:Psi_integral}
\Psi(\tau) =\mb S(\tau)  \mb U(\bs v, T) + \int_{0}^{\tau} \bs S(\tau - \tau') \bs N(\Psi(\tau')) d\tau' \quad \text{for} \quad \tau \geq 0.
\end{align}
In the following, \eqref{eq:Psi_integral} will be studied in the function space $\mc X$ given by
\[\mc X := \left \{ \Psi \in  C([0,\infty),\mc H): \sup_{\tau > 0} e^{\mu_p \tau} \|\Psi(\tau)\| < \infty \right \}.\]
where the exponent $\mu_p$ was defined in Lemma \ref{Le:LinearTimeEvol}. 

\begin{remark}
We note that in the above formulation of the problem the particular properties
of the underlying function space $\mc H$ are hidden in the abstract setting.
From now on, the proofs  mainly rely on the estimates for the
nonlinearity (Lemma \ref{lemma:nonlinL2}) and
the semigroup on the stable and unstable subspaces (Lemma \ref{Le:LinearTimeEvol}).
Therefore, most of the subsequent analysis can be copied from \cite{DonSch12}.
 Hence, we will only sketch the proofs of the following results and
refer the reader to \cite{DonSch12} for the details of the calculations.
\end{remark}

\subsection{Global existence for corrected (small) initial data}

The main problem which has to be addressed first is the exponential growth of the semigroup on the unstable subspace.
As in \cite{DonSch12} we introduce a correction term and consider the fixed point problem
\begin{align}
\label{Eq:Mod_Integral}
\Psi = \mb K(\Psi, \bs U(\bs v, T))
\end{align}
where
\begin{align}
\label{def:K}
\begin{split}
\mb K(\Psi,\bs u)(\tau):= & \mb  S(\tau) (1-\bs P) \bs u - \int_{0}^{\infty} e^{\tau - \tau'} \bs P \bs {N}(\Psi(\tau')) d\tau'  \\
& +  \int_{0}^{\tau} \mb S(\tau - \tau') \bs {N}(\Psi(\tau')) d\tau'.
\end{split}
\end{align}
Note that $\Psi = \mb K(\Psi, \bs U(\bs v, T))$ corresponds to the original equation 
\eqref{eq:Psi_integral} for initial data modified by
\[  - \bs P \left[\bs U(\bs v, T) + \int_0^{\infty} e^{-\tau'} \bs N(\Psi(\tau')) d\tau'\right], \]
an element of the unstable subspace $\bs P\mc H$ depending on the solution itself.
As we will see, this correction forces decay of the solution.
In the following we restrict ourselves to a closed ball $\mc X_{\delta} \subset \mc X$ defined by
\[\mc X_\delta := \{ \Psi \in \mc X: \|\Psi  \|_{\mc X} \leq \delta\}\]
for $\delta > 0$.
Recall that $\bs U(\bs 0,1) = \bs 0$, such that by continuity $\|\bs U(\bs v,T)\|$ 
is small for $\bs v$ small and $T$ close to $1$. 

\begin{theorem}
\label{Th:GlobalEx_ModEq}
Let $\mc U \subset \mc H$ be a sufficiently small neighborhood of $\bf 0$. 
Then, for
any $\bs u\in \mc U$, there exists a unique $\Psi_{\bs u} \in \mc X$ which satisfies
\[ \Psi_{\bs u}=\mb{K}( \Psi_{\bs u}, \bs u). \]
Furthermore, the map $\mb \Psi: \mc U \to \mc X$ defined by $\mb \Psi(\bs u ) := \Psi_{\bs u}$
 is  Fr\'{e}chet differentiable at $\bs u =\bs 0$. 
In particular, $\mb \Psi(\mb U(\bs v,T))$ exists provided $\bs v \in \mathfrak H$ is 
sufficiently small and $T$ is sufficiently close to $1$.
\end{theorem}

\begin{proof}
In the following we refer the reader to \cite{DonSch12},
Section $4.4$, for the details of the calculations.

Using the results of Lemma \ref{lemma:nonlinL2} one immediately obtains 
\begin{align}
\label{Est:Nonlin_X}
\begin{split}
\|\bs {N}(\Psi(\tau)) \|  & \lesssim \delta^2 e^{-2 \mu_p \tau}, \\ 
\|\bs {N}(\Psi(\tau)) - \bs {N}(\Phi(\tau)) \|  &\lesssim \delta e^{- \mu_p \tau} \|\Psi(\tau) - \Phi(\tau) \| 
\end{split}
\end{align}
for $0 < \delta < 1$,  $\Phi, \Psi \in \mc X_{\delta}$ and all $\tau \geq 0$. 

Note, that the integrals in \eqref{def:K} exist as Riemann integrals over continuous functions 
for $(\Psi, \bs u) \in \mc X_{\delta} \times \mc H$.
We decompose $\mb K$ according to
\[\mb K(\Psi, \bs u)(\tau) = \bs P\mb K(\Psi, \bs u)(\tau)+(1-\bs P)\mb K(\Psi, \bs u)(\tau) \]
and show that $\mb K(\Psi, \bs u) \in \mc X_{\delta}$ for 
$\Psi \in \mc X_{\delta}$, $\| \bs u \| \leq \delta^2$, and $\delta$ sufficiently 
small. Using the estimates for the semigroup $\mb S$ on $\bs P\mc H$ and $(1 - \bs P) \mc H$, 
cf. Lemma \ref{Le:LinearTimeEvol}, together with \eqref{Est:Nonlin_X} it is easy to 
see that for $\| \bs u \| \leq \delta^2$ and $\tau \geq 0$ we have
\begin{align*}
\|\bs P\mb K(\Psi, \bs u)(\tau)\|  & \lesssim \delta^2 e^{-2\mu_p \tau},\\
\|(1-\bs P) \mb K(\Psi, \bs u)(\tau)\|  & \lesssim \delta^2 e^{-\mu_p \tau}
\end{align*}
which implies
$\|\mb K(\Psi, \bs u)(\tau) \| \leq \delta e^{-\mu_p \tau}$ provided $\delta>0$ is sufficiently small.
Continuity of $\mb K(\Psi, \bs u)$ with respect to $\tau$ follows essentially from strong continuity of 
the semigroup $\mb S$ and we conclude that $ \mb K(\Psi, \bs u) \in \mc X_{\delta}$.

To see that $\mb K(\cdot, \bs u)$ is contracting we again use Lemma \ref{Le:LinearTimeEvol} and \eqref{Est:Nonlin_X}
to infer that 
\begin{align*}
\|\bs P[\mb K(\Phi, \bs u)(\tau)-  \mb K(\Psi, \bs u)(\tau)]\|   & \lesssim  \delta  e^{-2 \mu_p \tau} \|\Phi - \Psi\|_{\mc X}\\
 \|(1-\bs P)[\mb K(\Phi, \bs u)(\tau)- \mb K(\Psi, \bs u)(\tau)]\| & \lesssim \delta e^{-\mu_p\tau} \|\Phi - \Psi\|_{\mc X}
\end{align*} 
for $\Psi, \Phi \in \mc X_{\delta}$ and $\tau \geq 0$. In particular, for $\delta$ sufficiently small, we obtain
\[ \|\mb K(\Phi,\bs u) - \mb K(\Psi, \bs u) \|_{\mc X} \leq \tfrac {1}{2} \|\Phi - \Psi\|_{\mc X}.\]

We apply the Banach fixed point theorem to infer that for any $\bs u \in \mc U$, the equation
\[\Psi = \mb K(\Psi, \bs u)\]
has a unique solution $\Psi_{\bs u}$ in the closed subset 
$\mc X_{\delta}$ provided $\mc U \subset \mc H$ is a sufficiently small neighborhood around $\mb 0$.  
Furthermore, standard arguments imply that this is in fact the unique solution in the whole space $\mc X$. 

The Banach fixed point theorem implies that the solution depends continuously on the initial data, i.e., the map
$\mb \Psi: \mc U \to \mc X$ is continuous. In particular, for
$\bs u, \bs {\tilde u} \in \mc U$ and the corresponding solutions $\bs \Psi(\bs u), \bs \Psi(\bs {\tilde u}) \in \mc X_{\delta}$  
it is easy to see that 
\begin{align}
\label{Est:Sol_Lipschitz}
\|\mb \Psi(\bs u) - \mb \Psi(\bs {\tilde u})\|_{\mc X} \lesssim \|\bs u-\bs {\tilde u}\|,
\end{align}
cf.~the proof of Theorem $4.7$ in \cite{DonSch12}.

In order to prove differentiability of $\bs \Psi(\bs u)$ at $\bs u = \bs 0$ we
define 
\[[\tilde D \mb \Psi(\bs 0)\bs {u}](\tau):=\mb S(\tau)(1-\bs P)\bs {u}\] and note that
 $\tilde D \mb \Psi(\bs 0): \mc H \to \mc X $ is linear and bounded. 
 We claim that $\tilde D \mb \Psi(\bs 0)$ is the  
Fr\'{e}chet derivative of $\mb \Psi$ at $\bs 0$. To prove this, we have to show that 
(recall that $\bs \Psi(\bs 0) = \bs 0$) 
\[\lim_{\bs {\tilde u} \to \bs 0} \frac{1}{\| \bs {\tilde u}\|}\|\mb \Psi(\bs {\tilde u}) - \tilde D \mb \Psi(\bs 0)\bs {\tilde u}\|_{\mc X} = 0.\]
For small $\bs {\tilde u}$ we have 
$\bs \Psi(\bs {\tilde u}) = \mb K(\bs \Psi(\bs {\tilde u}),\bs {\tilde u})$ and by definition we infer
\begin{align*}
\mb \Psi(\bs {\tilde u}) - \tilde D \mb \Psi(\bs 0)\bs {\tilde u} = &  \int_0^\tau \mb S(\tau - \tau') \bs N(\bs \Psi(\bs {\tilde u})(\tau'))d\tau' \\
 & - \int_0^{\infty} e^{\tau -\tau'} \bs P \bs N(\bs \Psi(\bs {\tilde u})(\tau'))d\tau' =:\mb G(\bs {\tilde u})(\tau)
\end{align*}
By using the decomposition 
\[\mb G(\bs {\tilde u})(\tau) = \bs P[\mb G(\bs {\tilde u})(\tau)] + (1- \bs P)[\mb G(\bs {\tilde u})(\tau)],\]
the estimates for the nonlinearity and the semigroup, as well as \eqref{Est:Sol_Lipschitz} we obtain 
\[ \|\mb G(\bs {\tilde u})\|_{\mc X} \lesssim \|\bs {\tilde u}\|^2 \]
which implies the claim.

Finally, since $\mb U: \mathfrak H \times \mc I\to \mc H$ is continuous 
and $\mb U(\mb 0,1)=\mb 0$ (Lemma \ref{initialdata_op}), it follows that $\mb U(\bs v,T) \in \mc U$ for all $\bs v$ sufficiently small
and $T$ sufficiently close to $1$.
\end{proof}
 
\subsection{Global existence for arbitrary (small) initial data}

We use the results of the previous section to 
obtain a global solution of the integral equation \eqref{eq:Psi_integral}.
In the following let $\mathfrak U \subset \mathfrak H$ be a sufficiently small open neighborhood of $\mb 0$ and
let $\mc J\subset \mc I$ be a sufficiently small open neigborhood of $1$.
For $(\bs v, T) \in \mathfrak U \times \mc J$, Theorem \ref{Th:GlobalEx_ModEq} yields the existence of a global solution
$\bs \Psi(\bs U(\bs v, T) )\in \mc X$ of the modified equation, which can be written as
\begin{align}
\label{Eq:Integral_CorrAbstract}
\bs \Psi(\bs U(\bs v, T))(\tau) =  &\mb S(\tau)  \bs U(\bs v, T) + \int_{0}^{\tau} \mb S(\tau - \tau')
 \bs N(\bs \Psi(\bs U(\bs v, T))(\tau')) d\tau' \nonumber \\
  & - e^{\tau} \bs F(\bs v, T)
\end{align}
for $\tau \geq 0$ where 
\begin{align*}
\bs F(\bs v, T) :=  \bs P \left[\bs U(\bs v, T) + \int_0^{\infty} e^{-\tau'} 
\bs N(\bs \Psi(\bs U(\bs v, T))(\tau')) d\tau'\right].
\end{align*}
Note that for $\bs v =\bs 0$ and $T=1$ we have $\bs U(\bs 0, 1)=\bs 0$ and thus, $\bs F(\bs 0,1) = \bs 0$. Hence,
\eqref{Eq:Integral_CorrAbstract} reduces to \eqref{eq:Psi_integral} and $\bs \Psi(\bs U(\bs 0,1)) = \bs 0$ solves the original equation.
In the following, we extend this to a neighbourhood of $(\bs 0, 1)$. 

\begin{lemma}
\label{Le:CorrectionTerm}
The function $\bs F: \mathfrak  U \times \mc J \subset \mathfrak H \times  \mc I \to \langle \bs  g \rangle$ is continuous. 
Furthermore, $\bs F(\bs 0,\cdot): \mc J \to  \langle \bs g \rangle$ is Fr\'echet differentiable at $1$ and 
$$D_T \bs F(\bs 0,T)\rst{_{T=1}} \lambda = \tfrac{2\lambda}{p-1} \kappa_p^{\frac{1}{p-1}} \bs g $$
for all $\lambda \in \R$.
As a consequence, for every $\bs v \in  \mathfrak  U$ there exists a $T \in \mc J$ such that 
$\bs F(\bs v,T) = \bs 0$.
\end{lemma}

\begin{proof}
To show continuity we rewrite the correction in a more abstract way by introducing operators
$\mb B: \mc X \to \mc H$ and $\mb N: \mc X \to \mc X$ defined by 
\[ \mb B \Psi := \int_0^{\infty} e^{-\tau} \Psi(\tau) d\tau, \quad  \mb N(\Psi)(\tau) := \bs N(\Psi(\tau)). \]
One can easily check that $\mb B$ is linear and bounded. Furthermore, the properties of the operator
$\bs N$ described in Lemma \ref{lemma:nonlinL2} imply the $\mb N$ is continuous, differentiable at $\mb 0 \in \mc X$,
and
\begin{align}
\label{Eq:DiffN} 
D\mb N(\mb 0)\Psi = \mb 0 \quad \text{for} \quad \Psi \in \mc X,
\end{align}
see also \cite{DonSch12}, proof of Lemma $4.9$. Thus, $\bs F$ can be written as a composition of continuous 
operators
\[\bs F(\bs v, T) = \bs P\left[ \bs U(\bs v,T) + \mb B \mb N(\bs \Psi(\bs U(\bs v,T))) \right].\]
For $\bs v = \bs 0$ fixed the right-hand side is differentiable with respect to $T$ at $T=1$,
see Lemma \ref{initialdata_op}, \eqref{Eq:DiffN} and Theorem \ref{Th:GlobalEx_ModEq}, and we obtain
\begin{align*}
D_T \bs F(\bs 0,T)\rst{_{T=1}} \lambda &=  \bs P D_T \bs U(\bs 0,T)\rst{_{T=1}}  \lambda \\
&\quad + 
\bs P \mb B  D \mb N(\bs 0) D \bs \Psi(\bs 0) D_T \mb U(\bs 0,T)\rst{_{T=1}} \lambda  \\
&=   \bs P D_T \bs U(\bs 0,T)\rst{_{T=1}} \lambda =  \tfrac{2\lambda}{p-1} \kappa_p^{\frac{1}{p-1}} \bs g.
\end{align*}

Now we prove the second claim using the fact that the range of $\bs F$ is contained in the one-dimensional 
vector space $\langle \bs g \rangle$. Let $I: \langle \bs g \rangle \to \R$ be the 
isomorphism given by $I(c \bs g) = c$ for $c \in \R$.
We define a real-valued, continuous function $f: \mathfrak U \times \mc J \to \R$ by $f = I \circ \bs F$. 
In particular, $f(\bs 0,\cdot): \mc J \to \R$ is continuous, differentiable at $1$, and 
$D_T f(\bs 0,T)\rst{_{T=1}} \neq \mb 0$.
Consequently, there exist $T^+,T^- \in \mc J$ such that $f(\bs 0,T^-)<0$ and $f(\bs 0,T^+)>0$. Since $f$ is continuous
in the first argument, we have $f(\bs v,T^+) >0$ and $f(\bs v,T^-) < 0$ for all 
$\bs v \in \tilde{\mathfrak U} \subset \mathfrak U$ provided $\tilde{\mathfrak U}$ is sufficiently small.
Consequently, by the intermediate value theorem we conclude that there exists a $T^*$ (depending on $\bs v$) 
such that $f(\bs v, T^*) = I(\bs F(\bs v, T^*)) = 0$ implying that $\bs F(\bs v, T^*) = \bs 0$.
\end{proof}

\begin{theorem}
\label{Th:Global_existence}
Let $\bs v \in \mathfrak H$ be sufficiently small. 
Then there exists a $T$ close to $1$ such that 
\begin{align*}
\Psi(\tau) =\mb S(\tau)  \mb U(\bs v, T) + \int_{0}^{\tau} \mb S(\tau - \tau') \bs N(\Psi(\tau')) d\tau'
\end{align*}
has a continuous solution $\Psi: [0, \infty) \to \mc H$ satisfying
\[\| \Psi (\tau)\| \leq \delta e^{-\mu_p \tau}\]
for all $\tau \geq 0$ and some $\delta \in (0,1)$.
Moreover, the solution is unique in $C([0,\infty),\mc H)$.
\end{theorem}
\begin{proof}
The existence of a unique solution in $\mc X_{\delta}$ is a direct consequence of
Theorem \ref{Th:GlobalEx_ModEq} and Lemma \ref{Le:CorrectionTerm}. 
The stated decay estimate follows from the definition of the space $\mc X_\delta$.
For the uniqueness of the solution in the space
$C([0,\infty),\mc H)$ we refer the reader to the proof of Theorem $4.11$ in \cite{DonSch12}.
\end{proof}

\subsection{Proof of the main theorem}
\begin{proof}
Choose $\varepsilon >0$ such that $\mu_p = \frac{2}{p-1} - \varepsilon >0$ and let the initial data $(f,g)$ satisfy
the assumptions of Theorem \ref{Th:Main}. We set
\[ v_1(r) := r g(r) - \tfrac{2 r}{p-1} \kappa_p^{\frac{1}{p-1}}, \quad v_2(r) := f(r) + rf'(r) -  \kappa_p^{\frac{1}{p-1}}.\]
By definition of the respective function spaces it is easy to see that 
\[ \|\bs v \|_{\mathfrak H} = \| (f,g) - (\psi^1(0,\cdot),\psi_t^1(0,\cdot)) \|_{\mc E^h(\frac32)}\]
where $\bs v = (v_1,v_2)$. Hence, the smallness condition in Theorem \ref{Th:Main} implies that $\bs v$ is so small that
Theorem  \ref{Th:Global_existence} applies. We infer that for a certain value $T >0$ close to one (depending on $\bs v$)
we obtain a unique mild solution
$\Psi \in C([0,\infty), \mc H)$ of \eqref{Eq:Abstract_NonlinearDiff}. Setting $\Phi(\tau):=\Psi(\tau + \log T)$ yields a 
unique mild solution $\Phi \in C((-\log T,\infty], \mc H)$ of 
\begin{align*}
\left \{ \begin{array}{l}
\frac{d}{d\tau}\Phi(\tau)=\bs L\Psi(\tau)  + \bs N(\Psi(\tau)) \mbox{ for }\tau>-\log T \\
\Phi(- \log T) =\bs U(\bs v,T)
\end{array} \right .
\end{align*}
satisfying
\[ \| \Phi(\tau) \| \leq C_{\varepsilon} e^{-\mu_p\tau} \]
for all $\tau \geq -\log T$ and a constant $C_{\varepsilon} > 0$.
By definition,
\[\Phi(\tau)(\rho) = (\phi_1(\tau,\rho),\phi_2(\tau, \rho))\]
is a solution of the original system \eqref{eq:nonlinear_firstorder_css}.
Using the identity \eqref{Eq:ReconstructField} we infer
\begin{align*}
\|(\psi(t,\cdot), &\psi_t(t,\cdot))-  (\psi^T(t,\cdot),\psi^T_t(t,\cdot))\|^2_{\mc E^h(T-t)} = \\
 & (T-t)^{-\frac{4}{p-1}} \int_0^{T-t} \Big (|\phi_1(-\textstyle{\log}(T-t),\tfrac{r}{T-t})|^2 \\
 &+ |\phi_2(-\textstyle{\log}(T-t),\tfrac{r}{T-t})|^2 \Big ) dr \\
 &+ (T-t)^{-\frac{2(p+1)}{p-1}} \int_0^{T-t} 
 \Big (|\partial_2 \phi_1(-\textstyle{\log}(T-t),\tfrac{r}{T-t})|^2 \\
 &+  |\partial_2 \phi_2(-\textstyle{\log}(T-t),\tfrac{r}{T-t})|^2 \Big )dr 
 \end{align*}
 and thus,
 \begin{align*}
 \|(&\psi(t,\cdot), \psi_t(t,\cdot))-  (\psi^T(t,\cdot),\psi^T_t(t,\cdot))\|^2_{\mc E^h(T-t)} = \\
 &= (T-t)^{\frac{p-5}{p-1}} \int_0^{1} \Big (|\phi_1(-\log(T-t),\rho)|^2 + |\phi_2(-\log(T-t),\rho)|^2 \Big )d\rho \\
 &\quad +(T-t)^{- \frac{p+3}{p-1}} \int_0^{1} \Big (|\partial_{\rho} \phi_1(-\log(T-t),\rho)|^2 \\
 &\quad + 
 |\partial_{\rho} \phi_2(-\log(T-t),\rho)|^2 \Big )d\rho \\
& \leq (T-t)^{- \frac{p+3}{p-1}}  \| \Phi(-\textstyle{\log}(T-t))\|^2 \leq C_{\varepsilon} (T-t)^{- \frac{p+3}{p-1} + \frac{4}{p-1} - 2\varepsilon}
\end{align*}
which implies the claimed estimate.
\end{proof}


\bibliography{wave}
\bibliographystyle{plain}
\end{document}